\documentclass[12pt,lecno]{amsart}
\usepackage{indentfirst,amssymb,amsmath,amsthm}
\newtheorem{theorem}{Theorem}[section]

\newtheorem{lemma}{Lemma}[section]
\newtheorem{remark}{Remark}[section]

\newtheorem{example}{Example}[section]

\newcommand{\be}{\begin{equation}}
\newcommand{\ee}{\end{equation}}
\newcommand{\bea}{\begin{eqnarray}}
\newcommand{\eea}{\end{eqnarray}}
\newcommand{\beas}{\begin{eqnarray*}}
\newcommand{\eeas}{\end{eqnarray*}}

\begin{document}

\setcounter{page}{1} \setlength{\unitlength}{1mm}\baselineskip
.58cm \pagenumbering{arabic} \numberwithin{equation}{section}

\title[ $f(\mathcal{R})$-gravity ]
{Characterizations of Perfect fluid spacetimes obeying $f(\mathcal{R})$-gravity equipped with different gradient solitons }

\author[  ]
{ Krishnendu De$^{*}$ , Young Jin Suh and Uday Chand De }

\address
 {$^{*}$  Department of Mathematics,
 Kabi Sukanta Mahavidyalaya,
The University of Burdwan.
Bhadreswar, P.O.-Angus, Hooghly,
Pin 712221, West Bengal, India. ORCID iD: https://orcid.org/0000-0001-6520-4520}
\email{krishnendu.de@outlook.in }

\address{ Department of Mathematics\\
 Kyungpook National University\\
Daegu-41566, South Korea}
\email{yjsuh@knu.ac.kr}

\address
{Department of Mathematics, University of Calcutta, West Bengal, India. ORCID iD: https://orcid.org/0000-0002-8990-4609}
\email{uc$_{-}$de@yahoo.com}
\footnotetext {$\bf{2020\ Mathematics\ Subject\ Classification\:}.$ 83C05; 83D05; 53C50.
\\ {Key words: $f(\mathcal{R})$-gravity; Perfect fluids; $\eta$-Ricci solitons; Einstein Solitons;  $m$-quasi Einstein solitons.\\
\thanks{$^{*}$ Corresponding author}
}}
\maketitle

\vspace{1cm}

\begin{abstract}

The prime object of this article is to study the perfect fluid spacetimes obeying $f(\mathcal{R})$-gravity, when $\eta$-Ricci solitons, gradient $\eta$-Ricci solitons, gradient Einstein Solitons and gradient $m$-quasi Einstein solitons are its metrics. At first, the existence of the $\eta$-Ricci solitons is proved by a non-trivial example. We establish conditions for which the $\eta$-Ricci solitons are expanding, steady or shrinking. Besides, in the perfect fluid spacetimes obeying $f(\mathcal{R})$-gravity, when the potential vector field of $\eta$-Ricci soliton is of gradient type, we acquire a Poisson equation. Moreover, we investigate gradient $\eta$-Ricci solitons, gradient Einstein Solitons and gradient $m$-quasi Einstein solitons in $f(\mathcal{R})$-gravity, respectively. As a result, we establish some significant theorems about dark matter era.
\end{abstract}

\maketitle

\section{Introduction}

A Lorentzian manifold is a specific type of semi-Riemannian manifold that has the Lorentzian metric $g$. A time-oriented 4-dimensional Lorentzian manifold $\mathcal{M}$ with signature $(-,+,+,+)$ is called spacetime.
\par

Einstein field equations (briefly, $\mathbf{EFE}$) is unfit to outline the late-time inflation of the universe without assuming the presence of a couple of undetectable components that should supply the cause for the dark matter and the dark energy. This is the major motivation for extending it to obtain a few of higher-order field equations of gravity. It's worth noting that $f(\mathcal{R})$-gravity is an intrinsic extension, rather than a generalization, of Einstein's gravity, in which the Hilbert-Einstein action of the gravitational field is replaced by a function $f(\mathcal{R})$, where $\mathcal{R}$ is the Ricci scalar. The preceding theory was introduced by Buchdahl in 1970 \cite{hab}, and it achieved viability with Starobinsky's \cite{aas} investigations of cosmic inflation. Various useful functional representations of $f(\mathcal{R})$ have been proposed in the literature, for more details see (\cite{cap}, \cite{cap1}, \cite{cap2}, \cite{ade}, \cite{ade1}, \cite{kde}, \cite{kde2}, \cite{kde1}).\par

In the general theory of relativity (briefly, $\mathbf{GR}$), the $\mathbf{EFE}$
\begin{equation}\label{a1}
    \mathcal{S}-\frac{\mathcal{R}}{2}g=\kappa^2 T,
\end{equation}
entails that $T$, the energy momentum tensor has vanishing divergence, where $\mathcal{S}$ is the Ricci tensor, $\mathcal{R}$ is the Ricci scalar and $\kappa =\sqrt{8\pi G}$, $G$ represents Newton's gravitational constant.

\par
It is well known that for a perfect fluid spacetime (briefly, $\mathcal{PFS}$) the energy momentum tensor (or, stress energy tensor) $T$ is given by,
\begin{equation}
\label{a2}
T=(\sigma+p)\mathcal{B} \otimes \mathcal{B}+p g,
\end{equation}
where $p$ being the isotropic pressure, $\sigma$ indicates the energy density of the $\mathcal{PFS}$ and $\rho$ is the velocity vector field which is defined by $g(E_1, \rho)=\mathcal{B}(E_1)$ for all $E_1$. In this $\mathcal{PFS}$, $\rho$ is a unit timelike vector field, that is, $g(\rho,\rho)=-1$.\par

The study of solitons on semi-Riemannian manifolds is a fascinating topic in differential geometry as well as in physics. Ricci solitons and their generalizations have seen rapid growth in recent years, providing new approaches for studying the topology and geometry of Riemannian manifolds. Because of their connection to general relativity, there has been a urge to study Ricci solitons and its generalizations in many geometrical contexts.\par

The concept of Ricci soliton is an intrinsic generalization of Einstein's metric. On a semi-Riemannian manifold $\mathcal{M}$, a Ricci soliton of a semi-Riemannian metric $g$ is defined by
\begin{equation}\label{a3}
\pounds_ {X_1}  g  + 2\beta_2 g + 2\mathcal{S} = 0,
\end{equation}
$\pounds$ being the Lie-derivative, a constant and a smooth vector field are denoted by $\beta_2$ and $X_1$, respectively, while the Ricci tensor is denoted by $\mathcal{S}$. The Ricci flow equation which was presented by Hamilton\cite{ham} is defined by
\begin{equation}\label{a4}
\frac{\partial}{\partial t}g = -2\mathcal{S}.
\end{equation}
The special solutions of the Ricci flow equations are known as Ricci solitons. In $\mathcal{PFS}$ with torse-forming vector field, Venkatesha and Kumara \cite{ven1} studied Ricci solitons. Ricci solitons have been investigated by several authors ( \cite{des2}, \cite{wan1}) and many others.\par

Real hypersurfaces in a complex space form was studied by Cho and Kimura\cite{cho} and have expanded the notion of Ricci solitons to $\eta$-Ricci solitons. The $\eta$-Ricci soliton on $(\mathcal{M},g)$  satisfies
\begin{equation}\label{a5}
\pounds_{X_1} g + 2\beta_2 g  + 2\beta_1 \eta\otimes\eta + 2\mathcal{S} = 0,
\end{equation}
$\beta_1$ being a constant and $\eta$ is a 1-form. If $\beta_1 = 0$, then it recovers Ricci solitons which are named proper for $\beta_1 \neq 0$.\par

The soliton turns into a gradient $\eta$-Ricci soliton if  $X_1$, the potential function is the gradient of a smooth function $f:\mathcal{M}\rightarrow \mathbb{R}$ . Hence the foregoing equation (\ref{a5}) transforms to
\begin{equation}\label{a6}
Hess\; \psi + \mathcal{S} + \beta_2 g + \beta_1 \eta\otimes \eta = 0,
\end{equation}
$Hess\; \psi$ being the Hessian of $\psi$. The $\eta$-Ricci soliton is called steady, shrinking or expanding according  as  $\beta_2 = 0$ , $\beta_2 < 0$ or $\beta_2 > 0$, respectively.\par

Many recent studies on $\eta$-Ricci and gradient $\eta$-Ricci solitons in the setting of contact and Riemannian manifolds have been published. In the examination of cosmological models such as $\mathcal{PFS}$, geometric flows are also initiated. Blaga explored the $\eta$-Ricci soliton in $\mathcal{PFS}$ and obtained the Poisson equation (\cite{blaga2}, \cite{bla3}).\par

If the semi-Riemannian metric $g$ obeys
\begin{equation}\label{a7}
Hess\; \psi + \mathcal{S} + (\beta_2  -\frac{\mathcal{R}}{2})g= 0,
\end{equation}
then it is called a gradient Einstein soliton \cite{cat1}. $\psi$ is called the Einstein potential function. The gradient Einstein soliton is called trivial \cite{cat1}, if $\psi$ is constant. This soliton is named shrinking, steady or expanding according as $\beta_2 < 0$, $\beta_2 = 0$ or $\beta_2 > 0$, respectively. Few interesting results of this solitons have been established in (\cite{cat}, \cite{de2}).\par

For a smooth function $\psi:\mathcal{M}\rightarrow \mathbb{R}$, if the semi-Riemannian metric $g$ fulfills
\begin{equation}\label{a8}
\mathcal{S}+Hess\; \psi=\frac{1}{m}d\psi \otimes d\psi +\beta g,
\end{equation}
where $\beta$ is a constant, $0 < m\leq \infty $ being the integer and $\otimes$ denote the tensor product, then it is named $m$-quasi Einstein metric \cite{br1}. Here, $\psi$ is called the $m$-quasi Einstein potential function \cite{br1}. If $m = \infty$, the preceding equation (\ref{a8}) turns into a gradient Ricci soliton. Few fundamental classifications of the above stated metrics was presented by He et al. \cite{he}. Also, a few characterizations of these solitons have been characterized (in details) in (\cite{de1}, \cite{hu}).\par

Although the majority of the researches of solitons have been worked out in Riemannian category, the Ricci and gradient Ricci solitons have been investigated in the Lorentzian setting (\cite{bar1}, \cite{bat}, \cite{bro},  \cite{on}). So far in the existing literature, few fascinating results on Ricci solitons, gradient Ricci and gradient Yamabe solitons in $\mathcal{PFS}$ on $\mathbf{GR}$ theory have been studied in (\cite{blaga2}, \cite{de}). In this article, we intend to investigate the $\eta$-Ricci solitons, gradient $\eta$-Ricci solitons, gradient Einstein Solitons and gradient $m$-quasi Einstein solitons on $\mathcal{PFS}$ obeying $f(\mathcal{R})$-gravity.\par
We layout our article as:\par
In Section $2$, the existence of $\eta$-Ricci solitons is established. In next Section, we deduce the expression of $\mathcal{S}$ in a $\mathcal{PFS}$ with constant $\mathcal{R}$ satisfying $f(\mathcal{R})$-gravity. Section $4$ and $5$ deal with $\eta$-Ricci solitons and gradient $\eta$-Ricci solitons, respectively, on a $\mathcal{PFS}$ obeying $f(\mathcal{R})$-gravity. We study the properties of $\mathcal{PFS}$ fulfilling $f(\mathcal{R})$-gravity equipped with gradient Einstein solitons in Section $6$. Next, we consider a gradient $m$-quasi Einstein solitons. The study is concluded with a discussion.

\section{Existence of $\eta$-Ricci solitons in spacetimes}
Here we construct a non-trivial example of spacetime to establish the existence of $\eta$-Ricci solitons.\par

\;\;\;\;\;\;In a Lorentzian manifold $\mathbb{R}^{4}$ of dimension 4, let the Lorentzian metric $g$ be given by
\begin{eqnarray}\label{26.1}
    ds^{2}&=&g_{ij}d\mathrm{w}^{i}d\mathrm{w}^{j}
    =-(d\mathrm{w}_4)^{2}+e^{2\mathrm{w}_4}[(d\mathrm{w}_3)^{2}+ (d\mathrm{w}_2)^2+(d\mathrm{w}_1)^{2}],
\end{eqnarray}
where $i,j=1,2,3,4$.\par
Utilizing (\ref{26.1}), we can write
\begin{equation}\label{26.2}
    g_{11}=e^{2\mathrm{w}_4},g_{22}=e^{2\mathrm{w}_4},\ g_{33}=e^{2\mathrm{w}_4}, g_{44}=-1.
\end{equation}

Making use of (\ref{26.2}), the non-vanishing components of Christoffel symbol, the curvature tensor $K$ and the Ricci tensor $\mathcal{S}$ are obtained as
$$\Gamma^{4}_{11}=\Gamma^{4}_{22}=\Gamma^{4}_{33}=e^{2\mathrm{w}_4},\;\;\Gamma ^{1}_{14}= \Gamma ^{2}_{24}=\Gamma ^{3}_{34}=1,$$

$$K_{1441}=K_{2442}=K_{3443}=e^{2\mathrm{w}_4},\;\;\ K_{1221}=K_{1331}=K_{2332}=-e^{4\mathrm{w}_4}.$$

$$\mathcal{S}_{11}=\mathcal{S}_{22}=\mathcal{S}_{33}=-3e^{2\mathrm{w}_4},\;\;\ \mathcal{S}_{44}=3.$$

Hence, the Ricci scalar $\mathcal{R}=-12$.\par
Let $\eta$, the 1-form be defined by $$\eta(X_1)=g(X_1,\rho)$$ for any $X_1$. Now if we take $X_1=\rho$ in (\ref{a5}), then the $\eta$-Ricci soliton takes the form

\begin{equation}\label{aa6}
2[ g_{ii}+\eta_{i}\otimes \eta_{i}] =- 2 \mathcal{S}_{ii} - 2\beta_2 g_{ii} -2\beta_1 \eta_{i}\otimes \eta_{i},
\end{equation}
for $i=1,2,3,4.$\par

Thus the data $(g,\rho,\beta_1,\beta_2)$ is an $\eta$-Ricci soliton on $(\mathbb{R}^{4},g)$  where $\beta_1=-1$ and $\beta_2=2$ and the soliton is expanding.\par

\section{ Perfect fluid spacetimes obeying $f(\mathcal{R})$-gravity}
Now, we deal with 4-dimensional $\mathcal{PFS}$ satisfying $f(\mathcal{R})$-gravity. We now consider $\mathcal{H}$, a modified Einstein-Hilbert action term
\begin{equation}\label{z2}
   \mathcal{H} =\frac{1}{\kappa^2}\int \mathcal{L} _m \sqrt{(-g)}d^{4}x+\frac{1}{\kappa^2}\int f(\mathcal{R}) \sqrt{(-g)}d^{4}x,
\end{equation}
where $\mathcal{L}_m$ is the matter Lagrangian density of the scalar field. Here,
\begin{equation}\label{z3}
   T_{ab}=\frac{-2\delta (\sqrt{-g})\mathcal{L}_m}{\sqrt{-g}\delta ^{ab}},
\end{equation}
where $T_{ab}$ are the components of the energy momentum tensor or the stress-energy tensor $T$.\par
In local coordinate system, in a $\mathcal{PFS}$, $T_{ab}$ is defined by
\begin{equation}\label{zz3}
   T_{ab}= (\sigma+p)\mathcal{B}_a \mathcal{B}_b+p g_{ab},
\end{equation}
where $\mathcal{B}_a$ is a unit time like vector.

Assuming $\mathcal{L}_m$ completely depends on $g_{ab}$ (the components of the metric tensor $g$), the variation of action of (\ref{z2}) produces the field equations of $f(\mathcal{R})$-gravity
\begin{eqnarray}
f_\mathcal{R}(\mathcal{R})\mathcal{R}_{ab}&-& \frac{1}{2}f(\mathcal{R})g_{ab}+(g_{ab}\Box-\nabla_a \nabla_b)f_\mathcal{R} (\mathcal{R}) \nonumber \\
   &&= \kappa^2 T_{ab},\label{z4}
\end{eqnarray}
where $\mathcal{R}_{ab}$ are the local components of the Ricci tensor $\mathcal{S}$ and $\Box\equiv\nabla_c \nabla^c$ is called d'Alembert operator, $\nabla_a$ denotes the covariant derivative.\par 
Taking constant Ricci scalar, the field equation (\ref{z4}) gives the form
\begin{equation}
\mathcal{R}_{ab}- \frac{\mathcal{R}}{2}g_{ab}=\frac{\kappa^2}{f_\mathcal{R}(\mathcal{R})} T_{ab}^{eff},\label{z7}
\end{equation}
where
\begin{equation*}
    T_{ab}^{eff}=T_{ab}
    +\frac{f(\mathcal{R})-\mathcal{R}f_\mathcal{R}(\mathcal{R})}{2\kappa^2}g_{ab}.
\end{equation*}
In case of constant $\mathcal{R}$, we acquire
\begin{align}\label{g5}
R_{ab}-\frac{\mathcal{R}}{2}g_{ab}=\frac{\kappa^2}{\mathcal{F}_\mathcal{R}(\mathcal{R})}T_{ab}
+\big[\frac{(\mathcal{F}(\mathcal{R})-\mathcal{R}\mathcal{F}_\mathcal{R}(\mathcal{R})}{2\mathcal{F}_\mathcal{R}(\mathcal{R})}g_{ab}\big].
\end{align}	
Using (\ref{zz3}) in (\ref{g5}), we acquire
\begin{align}\label{h1}
R_{ab}-\frac{\mathcal{R}}{2}g_{ab}=\frac{\kappa^2}{\mathcal{F}_\mathcal{R}(\mathcal{R})}\{(p+\sigma)\mathcal{B}_{a}\mathcal{B}_{b}
+pg_{ab}\}
+\big[\frac{(\mathcal{F}(\mathcal{R})-\mathcal{R}\mathcal{F}_\mathcal{R}(\mathcal{R})}{2\mathcal{F}_\mathcal{R}(\mathcal{R})}g_{ab}\big].
\end{align}	
Multiplying both sides of the previous equation by $g^{ab}$ gives
\begin{align}\label{h2}
\mathcal{R}-2\mathcal{R}=\frac{\kappa^2}{\mathcal{F}_\mathcal{R}(\mathcal{R})}\{-(p+\sigma)+4p\}
+2\big[\frac{(\mathcal{F}(\mathcal{R})-\mathcal{R}\mathcal{F}_\mathcal{R}(\mathcal{R})}
{\mathcal{F}_\mathcal{R}(\mathcal{R})}\big],
\end{align}	
which implies
\begin{align}\label{h3}
\mathcal{R}=\frac{1}{\mathcal{F}_\mathcal{R}(\mathcal{R})}\{\kappa^{2}(3p-\sigma)+2\mathcal{F}(\mathcal{R})\}.
\end{align}
Putting the value of $\mathcal{R}$ in (\ref{h1}) yields

\begin{align}\label{h4}
R_{ab}=\frac{2\kappa^2 p+\mathcal{F}(\mathcal{R})}{2\mathcal{F}_\mathcal{R}(\mathcal{R})} g_{ab}+\frac{\kappa^2 (p+\sigma)}{\mathcal{F}_\mathcal{R}(\mathcal{R})}\mathcal{B}_{a}\mathcal{B}_{b}.
\end{align}
\begin{theorem}
For constant Ricci scalar, in $\mathcal{F}(\mathcal{R})$-gravity theory fulfilling a perfect fluid spacetime solution, the Ricci tensor $R_{ab}$ is of the form (\ref{h4}).
\end{theorem}
In index free notation the Ricci tensor is given by
\begin{equation}\label{b7}
    \mathcal{S}(E_1,F_1)=\alpha_1 \mathcal{B}(E_1)\mathcal{B}(F_1)+\alpha_2 g(E_1,F_1),
\end{equation}
where $\alpha_1=\frac{\kappa^2 (p+\sigma)}{f_\mathcal{R} (\mathcal{R})}$ and $\alpha_2=\frac{2\kappa^2 p+f(\mathcal{R})}{2f_\mathcal{R} (\mathcal{R})}$. From the last equation the Ricci operator $Q$ (defined by $g(QE_1,F_1)=\mathcal{S}(E_1,F_1))$ is written by
\begin{equation}\label{b8}
Q(E_1)=\alpha_1 \mathcal{B}(E_1)\rho+\alpha_2 E_1,
\end{equation}
where $\rho$ is the velocity vector field corresponding to the 1-form $\mathcal{B}$.\par
Furthermore, $p$ and $\sigma$ are interconnected by a state equation of the form $p = p(\sigma )$ and the $\mathcal{PFS}$ is named isentropic. Moreover, the $\mathcal{PFS}$ is called stiff matter fluid, if $p = \sigma$. The $\mathcal{PFS}$ is said to be the dust matter fluid if $p = 0$, the dark matter era if $p+\sigma =0$ and the radiation era if $p =\frac{\sigma}{3}$.

\section{  $\eta$-Ricci solitons on perfect fluid spacetimes obeying $f(\mathcal{R})$-gravity}

$\;\;\;\;$ Suppose a $\mathcal{PFS}$ with constant $\mathcal{R}$ obeying $f(\mathcal{R})$-gravity, admits an $\eta$-Ricci soliton defined by $(\ref{a5})$, where the 1-form $\eta$ is identical with the 1-form $\mathcal{B}$ of the $\mathcal{PFS}$, that is, $\eta(E_{1})=\mathcal{B}(E_{1})=g(E_{1},\rho)$, where $\rho$ is a unit timelike vector. Then the equation $(\ref{a5})$ can be rewritten as
\begin{equation}\nonumber
\pounds_{X_1} g + 2\beta_2 g  + 2\beta_1 \mathcal{B}\otimes\mathcal{B} + 2\mathcal{S} = 0.
\end{equation}
Making use of the explicit form of the Lie derivative, we get
\begin{equation}
\label{c1}
\mathcal{S}(E_1,F_1)=-\frac{1}{2}[g(\nabla_{E_1} \rho,F_1)+g(E_1,\nabla_{F_1} \rho)]-\beta_2 g(E_1,F_1)-\beta_1 \mathcal{B}(E_1)\mathcal{B}(F_1).
\end{equation}
Contracting the previous equation, we acquire
\begin{equation}\label{c2}
    \mathcal{R}=-div \rho -4\beta_2 +\beta_1,
\end{equation}
where `div' denotes divergence.
Again, contracting the equation (\ref{b7}), we obtain
\begin{equation}\label{c3}
    \mathcal{R}=-\alpha_1 +4\alpha_2.
\end{equation}
Making use of the last two equations, we get
\begin{equation}\label{c4}
    -\alpha_1 +4\alpha_2=-div \rho -4\beta_2+\beta_1.
\end{equation}
Comparing (\ref{b7}) and (\ref{c1}), we acquire
\begin{equation}\nonumber
\alpha_1 \mathcal{B}(E_1)\mathcal{B}(F_1)+\alpha_2 g(E_1,F_1)=-\frac{1}{2}[g(\nabla_{E_1} \rho,F_1)+g(E_1,\nabla_{F_1} \rho)]-\beta_2 g(E_1,F_1)-\beta_1 \mathcal{B}(E_1)\mathcal{B}(F_1).
\end{equation}
Since $\rho$ is a unit timelike vector, we have $g(\nabla_{\rho} \rho,\rho)=0$. Using this result and putting $E_1=F_1=\rho$ in the last equation, we infer that
\begin{equation}\label{c5}
    \alpha_1 -\alpha_2=\beta_2-\beta_1.
\end{equation}
Utilizing the foregoing equation in (\ref{c4}), we have
\begin{equation}\label{c6}
    \beta_2= -\alpha_2 - \frac{div \rho}{3}.
\end{equation}
Let us choose that the potential vector field has vanishing divergence. Therefore, equation (\ref{c6}) reveals that
\begin{equation}\label{c7}
    \beta_2= -\alpha_2=-\frac{2\kappa^2 p+f(\mathcal{R})}{2f_\mathcal{R} (\mathcal{R})},
\end{equation}
where we have used equation (\ref{b7}).\par
Thus, we can write the subsequent:
\begin{theorem}
\label{thm3.1}
If $(g,\rho,\beta_2,\beta_1 )$ is an $\eta$-Ricci soliton in a $\mathcal{PFS}$ with constant $\mathcal{R}$ fulfilling $f(\mathcal{R})$-gravity and the 1-form $\eta$ is identical with the 1-form $\mathcal{B}$ of the $\mathcal{PFS}$ and the potential vector field $\rho$ has vanishing divergence, then the soliton is steady if $p=-\frac{f(\mathcal{R})}{2\kappa^2}$, expanding for $p<\frac{f(\mathcal{R})}{2\kappa^2}$ and shrinking for $p>\frac{f(\mathcal{R})}{2\kappa^2}$.
 \end{theorem}
Now solving the equations (\ref{c4}) and (\ref{c5}) after putting the value of $\alpha_1, \alpha_2$, we acquire
\begin{equation}\label{e8}
\left\{ \begin{array}{lcl}
\beta_2 =-\frac{div \rho}{3}-\frac{2\kappa^2 p+f(\mathcal{R})}{2f_\mathcal{R} (\mathcal{R})}  
\\\\
\beta_1=-\frac{div \rho}{3}- \frac{\kappa^2 (p+\sigma)}{f_\mathcal{R} (\mathcal{R})}.
\end{array}\right.
\end{equation}
Hence, from the foregoing equation, we infer that
$$\Delta (\psi)=div (grad \psi)=div \rho=-3(\beta_1+ \frac{\kappa^2 (p+\sigma)}{f_\mathcal{R} (\mathcal{R})}).$$
Therefore, we can state:
\begin{theorem}
\label{thm3.2}
Let the $\mathcal{PFS}$ with constant $\mathcal{R}$ obeying $f(\mathrm{R})$-gravity admit a $\eta$- Ricci solitons. If the 1-form $\eta$ is identical with the 1-form $\mathcal{B}$ of the $\mathcal{PFS}$ and  $\eta$ is the $g$-dual 1-form of the gradient vector field $\rho =$ grad$(\psi)$, then the Poisson equation obeys by $\psi$ is
$$\Delta (\psi)=-3(\beta_1+ \frac{\kappa^2 (p+\sigma)}{f_\mathcal{R} (\mathcal{R})}).$$
\end{theorem}
\begin{example}
In dark matter era with constant $\mathcal{R}$ satisfying $f(\mathcal{R})$-gravity, the $\eta$-Ricci solitons $(g,\rho,\beta_1,\beta_2)$ is given by
\begin{equation*}
    \left\{ \begin{array}{lcl}
\beta_2 =-\frac{div \rho}{3}-\frac{2\kappa^2 p+f(\mathcal{R})}{2f_\mathcal{R} (\mathcal{R})}  
\\\\
\beta_1=-\frac{div \rho}{3}.
\end{array}\right.
\end{equation*}
\end{example}
\begin{example}
In stiff matter fluid with constant $\mathcal{R}$ fulfilling $f(\mathcal{R})$-gravity, the $\eta$-Ricci solitons $(g,\rho,\beta_1,\beta_2)$ is given by
\begin{equation*}
    \left\{ \begin{array}{lcl}
\beta_2 =-\frac{div \rho}{3}-\frac{2\kappa^2 p+f(\mathcal{R})}{2f_\mathcal{R} (\mathcal{R})}  
\\\\
\beta_1=-\frac{div \rho}{3}- \frac{2 \kappa^2 p}{f_\mathcal{R} (\mathcal{R})}.
\end{array}\right.
\end{equation*}
\end{example}
\begin{example}
In a radiation era with constant $\mathcal{R}$ obeying $f(\mathcal{R})$-gravity, the $\eta$-Ricci solitons $(g,\rho,\beta_1,\beta_2)$ is given by
\begin{equation*}
    \left\{ \begin{array}{lcl}
\beta_2 =-\frac{div \rho}{3}-\frac{2\kappa^2 p+f(\mathcal{R})}{2f_\mathcal{R} (\mathcal{R})}  
\\\\
\beta_1=-\frac{div \rho}{3}- \frac{4 \kappa^2 p}{f_\mathcal{R} (\mathcal{R})}.
\end{array}\right.
\end{equation*}
\end{example}

\section{gradient $\eta$-Ricci solitons on perfect fluid spacetimes obeying $f(\mathcal{R})$-gravity}

\;\;\;\;Suppose that the soliton vector field $X_1$ of the $\eta$-Ricci soliton in a $\mathcal{PFS}$ with constant $\mathcal{R}$ satisfying $f(\mathcal{R})$-gravity is a gradient of some smooth function $\psi$, where the 1-form $\eta$ is identical with the 1-form $\mathcal{B}$ of the $\mathcal{PFS}$. Then equation (\ref{a6}) converts to
\begin{equation}\label{aa6}
Hess\; \psi + \mathcal{S} + \beta_2 g + \beta_1 \mathcal{B}\otimes \mathcal{B} = 0
\end{equation}
which implies
\begin{equation}
\label{3.3}
\nabla_{E_1}D\psi=- QE - \beta_2 E_1-\beta_1 \mathcal{B}(E_1)\rho,
\end{equation}
for all $E_1 \in \mathfrak{X_1}(\mathcal{M})$.
Because of the foregoing equation and the relation
\begin{equation}
\label{3.4}
R(E_1, F_1)D\psi=\nabla_{E_1} \nabla_{F_1}D\psi-\nabla_{F_1} \nabla_{E_1}D\psi-\nabla_{[E_1, F_1]}D\psi,
\end{equation}
we acquire
\begin{eqnarray}
\label{3.5}
R(E_1, F_1)D\psi&=&-[(\nabla_{E_1}Q)(F_1)-(\nabla_{F_1}Q)(E_1)]\nonumber \\&&
+\beta_1[(\nabla_{F_1} \mathcal{B})(E_1)\rho+\mathcal{B}(E_1)\nabla_{F_1} \rho-(\nabla_{E_1} \mathcal{B})(F_1)\rho-\mathcal{B}(F_1)\nabla_{E_1} \rho].
\end{eqnarray}
Now making use of the equation (\ref{b8}) we get
\begin{eqnarray}
\label{3.6}
&&(\nabla_{E_1}Q)(F_1)=E_1 (\alpha_2)F_1\nonumber \\&&
+E_1 (\alpha_1) \mathcal{B}(F_1)\rho+\alpha_1 \{(\nabla_{E_1} \mathcal{B})F_1 \rho+\mathcal{B}(F_1)\nabla_{E_1} \rho\}.
\end{eqnarray}
Utilizing (\ref{3.6}) in (\ref{3.5}), we obtain
\begin{eqnarray}
\label{3.7}
&&R(E_1, F_1)D\psi=-[E_1 (\alpha_2)F_1-F_1 (\alpha_2)E_1 +E_1 (\alpha_1) \mathcal{B}(F_1)\rho-F_1 (\alpha_1) \mathcal{B}(E_1)\rho\nonumber\\&&
+\alpha_1 \{(\nabla_{E_1} \mathcal{B})F_1 \rho+\mathcal{B}(F_1)\nabla_{E_1} \rho-(\nabla_{F_1} \mathcal{B})E_1 \rho-\mathcal{B}(E_1)\nabla_{F_1} \rho\}]\nonumber \\&&
+\beta_1[(\nabla_{F_1} \mathcal{B})(E_1)\rho+\mathcal{B}(E_1)\nabla_{F_1} \rho-(\nabla_{E_1} \mathcal{B})(F_1)\rho-\mathcal{B}(F_1)\nabla_{E_1} \rho].
\end{eqnarray}
Contracting the preceding equation, we find
\begin{eqnarray}
\label{3.8}
S(F_1, D\psi)&=&-[F_1 (\alpha_1)+\rho (\alpha_1) \mathcal{B}(F_1)-3 F_1 (\alpha_2)+\alpha_1 \{(\nabla_{\rho} \mathcal{B})F_1 -\mathcal{B}(F_1)div \rho\}]\nonumber \\&&
+\beta_1[(\nabla_\rho \mathcal{B})(F_1)+\mathcal{B}(F_1) div \rho],
\end{eqnarray}
where `$div$' being the divergence.
Again, from the equation (\ref{b7}), we infer
\begin{equation}
\label{3.9}
S(F_1, D\psi)=\alpha_1 \mathcal{B}(F_1)\mathcal{B}(D \psi)+\alpha_2 g(F_1, D \psi).
\end{equation}
Putting $F_1=\rho$ in equations (\ref{3.8}) and (\ref{3.9}) and then comparing, we acquire
\begin{equation}
\label{3.10}
(\alpha_2-\alpha_1) \rho( \psi)= [3 \rho (\alpha_2)+\alpha_1 div \rho]+\beta_1 div \rho.
\end{equation}
Let the scalar $\alpha_2$ and the potential function $\psi$ remain invariant under the velocity vector $ \rho$, that is, $\rho (\alpha_2)=0$ and  $\rho(\psi) =0$.
Hence, equations (\ref{3.10}) reveals that
\begin{equation}
\label{3.11}
(\alpha_1 +\beta_1 )div \rho=0,
\end{equation}
which entails that either $\alpha_1 +\beta_1 = 0$ or $div \rho =0$.\par
{\bf Case I}.
We suppose that $\alpha_1 +\beta_1= 0$ and $div \rho \ne 0$. Then putting the value of $\alpha_1$ from equation (\ref{b7}), we conclude that
\begin{equation}
\label{3.12}
\frac{\kappa^2 (p+\sigma)}{f_\mathcal{R} (\mathcal{R})}=\beta_{1}= constant,
\end{equation}
which implies that $p+\sigma=$constant.\par

{\bf Case II}. We assume that  $\alpha_1 +\beta_1\neq 0$ and $div \rho =0$ which means the velocity vector field is conservative. Since a conservative vector field is always irrotational, we get the vorticity of the perfect fluid is zero.\par
Therefore, we can state:
\begin{theorem}
\label{thm4.1}
Let the $\mathcal{PFS}$ with constant $\mathcal{R}$ obeying $f(\mathcal{R})$-gravity admit a gradient $\eta$-Ricci soliton. If the 1-form $\eta$ is identical with the 1-form $\mathcal{B}$ of the $\mathcal{PFS}$ and the scalar $\alpha_2$ and the potential function $\psi$ remain invariant under the vector field $\rho$, then either the spacetime represents the state equation $p+\sigma=$constant, or the perfect fluid has vanishing vorticity.
\end{theorem}
\begin{remark}
If the constant of state equation vanishes, then we have $p+\sigma=0$, that is, the $\mathcal{PFS}$ represents the dark matter era.
\end{remark}

\section{gradient Einstein solitons on perfect fluid spacetimes obeying $f(\mathcal{R})$-gravity}

\;\;\;\;Suppose that the soliton vector field $X_1$ of the Einstein soliton in a $\mathcal{PFS}$ with constant $\mathcal{R}$ obeying $f(\mathcal{R})$-gravity is a gradient of some smooth function $\psi$. Then equation (\ref{a7}) converts to
\begin{equation}
\label{c3.3}
\nabla_{E_1}D\psi=- QE - (\beta_2 -\frac{\mathcal{R}}{2})E_1,
\end{equation}
for all $E_1 \in \mathfrak{X_1}(\mathcal{M})$.
Because of the foregoing equation and the relation
\begin{equation}
\label{c3.4}
R(E_1, F_1)D\psi=\nabla_{E_1} \nabla_{F_1}D\psi-\nabla_{F_1} \nabla_{E_1}D\psi-\nabla_{[E_1, F_1]}D\psi,
\end{equation}
we get
\begin{eqnarray}
\label{c3.5}
R(E_1, F_1)D\psi=-[(\nabla_{E_1}Q)(F_1)-(\nabla_{F_1}Q)(E_1)]
\end{eqnarray}
Now using the equation (\ref{b8}) we obtain
\begin{eqnarray}
\nonumber
&&(\nabla_{E_1}Q)(F_1)=E_1 (\alpha_2)F_1\nonumber \\&&
+E_1 (\alpha_1) \mathcal{B}(F_1)\rho+\alpha_1 \{(\nabla_{E_1} \mathcal{B})F_1 \rho+\mathcal{B}(F_1)\nabla_{E_1} \rho\}.
\end{eqnarray}
Making use of the above equation in (\ref{c3.5}), we infer that
\begin{eqnarray}
\label{c3.7}
&&R(E_1, F_1)D\psi=-[E_1 (\alpha_2)F_1-F_1 (\alpha_2)E_1 +E_1 (\alpha_1) \mathcal{B}(F_1)\rho-F_1 (\alpha_1) \mathcal{B}(E_1)\rho\nonumber\\&&
+\alpha_1 \{(\nabla_{E_1} \mathcal{B})F_1 \rho+\mathcal{B}(F_1)\nabla_{E_1} \rho-(\nabla_{F_1} \mathcal{B})E_1 \rho-\mathcal{B}(E_1)\nabla_{F_1} \rho\}].
\end{eqnarray}
Contracting the foregoing equation, we obtain
\begin{eqnarray}
\label{c3.8}
S(F_1, D\psi)&=&-[F_1 (\alpha_1)+\rho (\alpha_1) \mathcal{B}(F_1)\nonumber\\&&
-3 F_1 (\alpha_2)+\alpha_1 \{(\nabla_{\rho} \mathcal{B})F_1 -\mathcal{B}(F_1)div \rho\}],
\end{eqnarray}
where `$div$' being the divergence.\par
Again, from the equation (\ref{b7}), we infer
\begin{equation}
\label{3.99}
S(F_1, D\psi)=\alpha_1 \mathcal{B}(F_1)\mathcal{B}(D \psi)+\alpha_2 g(F_1, D \psi).
\end{equation}

Replacing $F_1$ by $\rho$ in equations (\ref{c3.8}) and (\ref{3.99}) and then comparing, we have
\begin{equation}
\label{c3.10}
(\alpha_2-\alpha_1) \rho (\psi)= [3 \rho (\alpha_2)+\alpha_1 div \rho].
\end{equation}
Let us choose the potential function $\psi$ and the scalar $\alpha_2$ remains invariant under $ \rho$, that is, $\rho (\psi)=0$ and  $\rho (\alpha_2)=0$.
Hence, equations (\ref{c3.10}) entails that
\begin{equation}
\label{c3.11}
\alpha_1 div \rho=0,
\end{equation}
which reveals that either $\alpha_1=0$ or $div \rho=0$.\par
{\bf Case I}.
We suppose that $\alpha_1=0$. Then putting the value of $\alpha_1$ from equation (\ref{b7}), we conclude that
\begin{equation}
\label{c3.12}
p+\sigma=0,
\end{equation}
that is, the dark matter era.\par
{\bf Case II}. We assume that $div \rho =0$, then similarly by previous theorem, we conclude that the $\mathcal{PFS}$ has vanishing vorticity.\par

Therefore, we can state:
\begin{theorem}
\label{thm4.1}
Let the $\mathcal{PFS}$ with constant $\mathcal{R}$ satisfying $f(\mathcal{R})$-gravity admit a Einstein soliton of gradient type. If the scalar $\alpha_2$ and the potential function $\psi$ remain invariant under the vector field $\rho$, then either the spacetime represents the dark matter era, or the perfect fluid has vanishing vorticity.
\end{theorem}

\section{Gradient $m$-quasi Einstein solitons on perfect fluid spacetimes obeying $f(\mathcal{R})$-gravity}
\;\;\;\;Here, we choose that in a $\mathcal{PFS}$ with constant $\mathcal{R}$ obeying $f(\mathcal{R})$-gravity is a gradient $m$-quasi Einstein solitons. At first, we establish the subsequent result
\begin{lemma}\label{lem1}
Every $\mathcal{PFS}$ with constant Ricci scalar in $f(\mathcal{R})$-gravity obeys the following:
\begin{eqnarray}\label{k2}
R(E_1,F_1)D \psi &=& (\nabla_{F_1}Q)E_1-(\nabla_{E_1} Q)F_1+\frac{\beta}{m}\{ F_1 (\psi)E_1-E_1 (\psi)F_1\}\nonumber\\&&
+\frac{1}{m}\{ E_1 (\psi)Q F_1-F_1 (\psi)Q E_1 \},
\end{eqnarray}
for all $E_1, \, F_1 \in \mathfrak{X_1}(\mathcal{M})$.
\end{lemma}
\begin{proof}
Let us assume that the $\mathcal{PFS}$ with constant $\mathcal{R}$ in $f(\mathcal{R})$-gravity admits $m$-quasi Einstein solitons. Then the equation (\ref{a8}) turns into
\begin{equation}\label{k3}
\nabla_{E_1}D \psi+Q E_1=\frac{1}{m}g(E_1,D \psi)D \psi+\beta E_1.
\end{equation}
Covariant derivative of (\ref{k3}) along $F_1$ yields that
\begin{eqnarray}\label{k4}
\nabla_{F_1}\nabla_{E_1}D \psi &=& -\nabla_{F_1}Q E_1+ \frac{1}{m} \nabla_{F_1}g(E_1,D \psi)D \psi\nonumber\\&& +\frac{1}{m} g(E_1,D \psi)\nabla_{F_1}D \psi+ \beta \nabla_{F_1}E_1.
\end{eqnarray}
Exchanging $E_1$ and $F_1$ in (\ref{k4}), we acquire
\begin{eqnarray}\label{k5}
\nabla_{E_1}\nabla_{F_1}D \psi &=& -\nabla_{E_1}Q  F_1+ \frac{1}{m}\nabla_{E_1}g(F_1,D \psi)D \psi\nonumber\\&& +\frac{1}{m} g(F_1,D \psi)\nabla_{E_1}D \psi +\beta \nabla_{E_1}F_1
\end{eqnarray}
and
\begin{equation} \label{k6}
\nabla_{[E_1,F_1]}D \psi = -Q[E_1,F_1]+ \frac{1}{m}g([E_1,F_1],D \psi)D \psi+\beta [E_1,F_1].
\end{equation}
Making use of  (\ref{k3})-(\ref{k6}) and together
with $R(E_1,F_1)D \psi  =\nabla_{E_1}\nabla_{F_1}D \psi-\nabla_{F_1}\nabla_{E_1}D \psi-\nabla_{[E_1,F_1]}D \psi$, we infer
\begin{eqnarray}
R(E_1,F_1)D \psi  &=& (\nabla_{F_1}Q)E_1-(\nabla_{E_1}Q)F_1  +\frac{\beta}{m}\{ F_1 (\psi)E_1-E_1 (\psi)F_1\}\nonumber\\&&
+\frac{1}{m} \{ E_1 (\psi)Q F_1-F_1 (\psi)Q E_1 \}.\nonumber
\end{eqnarray}
\end{proof}
Again, making use of the equation (\ref{b8}) we obtain
\begin{eqnarray}
&&(\nabla_{E_1}Q)(F_1)=E_1 (\alpha_2)F_1\nonumber \\&&
+E_1 (\alpha_1) \mathcal{B}(F_1)\rho+\alpha_1 \{(\nabla_{E_1} \mathcal{B})F_1 \rho+\mathcal{B}(F_1)\nabla_{E_1} \rho\}.
\end{eqnarray}
Utilizing the previous equation, (\ref{b8}) and the above Lemma, we get
\begin{eqnarray}
\label{kk6}
&&R(E_1, F_1)D\psi=-[E_1 (\alpha_2)F_1-F_1 (\alpha_2)E_1 +E_1 (\alpha_1) \mathcal{B}(F_1)\rho-F_1 (\alpha_1) \mathcal{B}(E_1)\rho\nonumber\\&&
+\alpha_1 \{(\nabla_{E_1} \mathcal{B})F_1 \rho+\mathcal{B}(F_1)\nabla_{E_1} \rho-(\nabla_{F_1} \mathcal{B})E_1 \rho-\mathcal{B}(E_1)\nabla_{F_1} \rho\}]\nonumber\\&&
+\frac{\beta}{m}\{F_1 (\psi)E_1-E_1 (\psi)F_1\}\nonumber\\&&
+\frac{1}{m}\{\alpha_1 E_1 (\psi)F_1+\alpha_2E_1 (\psi)\mathcal{B}(F_1)\rho-\alpha_1 F_1 (\psi)E_1-\alpha_2F_1 (\psi)\mathcal{B}(E_1)\rho\}.
\end{eqnarray}
Contracting the foregoing equation gives
\begin{eqnarray}
\label{k8}
S(E_1, D\psi)&=&-[F_1 (\alpha_1)+\rho (\alpha_1) \mathcal{B}(F_1)\nonumber\\&&
-3 F_1 (\alpha_2)+\alpha_1 \{(\nabla_{\rho} \mathcal{B})F_1 -\mathcal{B}(F_1)div \rho\}+\frac{3\beta}{m}E_1 (\psi) \nonumber\\&&
+\frac{1}{m}\{\alpha_1 E_1 (\psi)+\alpha_2\rho (\psi)\mathcal{B}(E_1)-4 \alpha_1 E_1 (\psi)+\alpha_2E_1 (\psi)\}.
\end{eqnarray}
Again, from the equation (\ref{b7}), we infer
\begin{equation}
\label{39.9}
S(F_1, D\psi)=\alpha_1 \mathcal{B}(F_1)\mathcal{B}(D \psi)+\alpha_2 g(F_1, D \psi).
\end{equation}
Replacing $F_1$ by $\rho$ in equations (\ref{k8}) and (\ref{39.9}) and then equating the values of $S(\rho, D\psi)$, we obtain
\begin{equation}
\label{k9}
(\alpha_1-\alpha_2+\frac{3\beta}{m}-\frac{3\alpha_1}{m}) \rho (\psi)=[3 \rho (\alpha_2)+\alpha_1 div \rho].
\end{equation}
If $\psi$ and $\alpha_2$ are invariant under $\rho$ then, from the previous equation, we find that $\alpha_1 =0$ and hence we infer that $\sigma +p=0$.\par
Therefore, we conclude the result as:
\begin{theorem}
\label{thm4.1}
 Let a $\mathcal{PFS}$ with constant $\mathcal{R}$ satisfying $f(\mathcal{R})$-gravity admit a $m$-quasi Einstein soliton of gradient type. If  $\psi$ and $\alpha_2$ are invariant under the velocity vector field $\rho$, then the spacetime represents the dark matter era.
\end{theorem}

\section{Conclusion}
Solitons are nothing more than waves in their truest form. Physically, waves propagate with minimum energy loss and preserve their speed and shape after colliding with another wave of the same type. In the treatment of initial-value problems for nonlinear partial differential equations describing wave propagation, solitons play a crucial role. It moreover explain the recurrence in the Fermi-Pasta-Ulam system \cite{fpu}.\par

Different metrics in $\mathcal{PFS}$ obeying $f(\mathcal{R})$-gravity are investigated in this article, namely $\eta$-Ricci solitons, gradient $\eta$-Ricci solitons, gradient Einstein Solitons, and gradient $m$-quasi Einstein solitons. We consider a $\mathcal{PFS}$ that admits $\eta$-Ricci solitons with constant $\mathcal{R}$ obeying $f(\mathcal{R})$-gravity and the 1-form $\eta$ is identical with the 1-form $\mathcal{B}$ of the $\mathcal{PFS}$ , and we obtain the condition for which the soliton is steady, expanding and shrinking, as well as we deduce a Poisson equation. Furthermore, if the spacetime admits gradient $\eta$-Ricci solitons and gradient Einstein Solitons under the identical conditions, then either the spacetime represents the dark matter era or the perfect fluid has vanishing vorticity. Finally, with the same condition, if the $\mathcal{PFS}$ admits a gradient $m$-quasi Einstein solitons, then the spacetime represents the dark matter era.

\section{Declarations}

\subsection{Competing interests}
The authors declare that they have no financial or personal relationships that may be perceived as influencing their work.
\subsection{Funding }
Not applicable.

\end{document}